\documentclass[11 pt]{article}

\usepackage{array}
\usepackage{tikz}
\usepackage[all]{xy}
\usepackage{enumerate}
\usepackage{mathrsfs}
\usepackage{bm} 
\usepackage{amsfonts}
\usepackage[all]{xy}
\usepackage{amsfonts,amscd,amssymb,amsmath,amsthm}
\usepackage{amsfonts,amssymb,amsmath,amsxtra,amscd,amsthm,eucal,graphicx,graphics}
\usetikzlibrary{topaths,calc}
\newtheorem{theorem}{Theorem}[section]
\newtheorem{lemma}[theorem]{Lemma}

\newtheorem{corollary}[theorem]{Corollary}

\theoremstyle{remark}
\newtheorem{remark}[theorem]{Remark}

\theoremstyle{definition}
\newtheorem{example}[theorem]{Example}
\newtheorem{definition}[theorem]{Definition}

\title{Bouquets, Vertex Covers and Edge Ideals}
 \author{Nursel Erey\footnote{Department of Mathematics and Statistics, Dalhousie University, Halifax, Nova Scotia, Canada, e-mail: nurselerey@gmail.com}}
\date{\today }

\providecommand{\keywords}[1]{\textbf{\textit{Keywords---}} #1}
\date{\vspace{-5ex}} 

\begin{document}
\maketitle

\begin{abstract}
We give a new combinatorial characterization of the big height of a squarefree monomial ideal leading to a new bound for the projective dimension of a monomial ideal.

\end{abstract}

\keywords{Big height, Edge ideal, Hypergraph, Minimal vertex cover}

\section{Introduction}
Let $\Bbbk$ be a field and $S=\Bbbk[x_1,\dots,x_n]$. The \emph{edge ideal} of a simple hypergraph $\mathcal{H}$ is a squarefree monomial ideal $I(\mathcal{H})\subseteq S$ given by
$$ I(\mathcal{H})=(x_{i_1}\dots x_{i_t} : \{x_{i_1},\dots,x_{i_t}\} \text{ is an edge of } \mathcal{H}). $$
 
A current research topic in commutative algebra is to express or bound the invariants of the minimal free resolution of the squarefree monomial ideal coming from a simple hypergraph in terms of its combinatorial properties (see, for example,  \cite{dao schweig} -- \cite{zheng}). Many authors introduced new graph parameters and notions in this context.

A subset $C$ of vertices of $\mathcal{H}$ is called a \emph{vertex cover} of $\mathcal{H}$ if every edge in $\mathcal{H}$ contains an element of $C$. A vertex cover $C$ is \emph{minimal} if no proper subset of $C$ is a vertex cover of $\mathcal{H}$. We write $\alpha_0^\prime (\mathcal{H})$ for the maximum possible cardinality of a minimal vertex cover of $\mathcal{H}$. There is a one-to-one correspondence between the minimal vertex covers of $\mathcal{H}$ and the \emph{minimal prime ideals} of $I(\mathcal{H})$ given by
$$C \text{ is a minimal vertex cover of } \mathcal{H} \Leftrightarrow (x_i : i\in C) \text{ is a minimal prime ideal of } I(\mathcal{H}). $$
Therefore the parameter $\alpha_0^\prime (\mathcal{H})$ coincides with the \emph{big height} of $I(\mathcal{H})$, which is the maximum height of the minimal prime ideals of $I(\mathcal{H})$. 
It is known that the big height of $I(\mathcal{H})$ is a lower bound for the projective dimension $\operatorname{pd}(S/I(\mathcal{H}))$ of $S/I(\mathcal{H})$, see, for example \cite[Corollary~3.33]{morey villarreal} for a proof. In fact, it is a sharp bound in the following sense.
\begin{theorem}\label{pd of scm is big height} \emph{\cite{faridi, morey villarreal}} For a simple hypergraph $\mathcal{H}$, the equality $\operatorname{pd}(S/I(\mathcal{H}))=\alpha_0^\prime (\mathcal{H})$ holds if $\mathcal{H}$ is sequentially Cohen-Macaulay, i.e., $S/I(\mathcal{H})$ is sequentially Cohen-Macaulay.
\end{theorem}
In this work, we will give a new characterization of $\alpha_0^\prime (\mathcal{H})$, or equivalently, the big height of $I(\mathcal{H})$, in terms of partial hypergraphs known as \emph{bouquets} of $\mathcal{H}$. This characterization is useful if one is interested in the structure of the minimal vertex covers while bounding the projective dimension via the big height.

Given a hypergraph $\mathcal{H}$ and a minimal vertex cover $A$ of size $i$, since $\alpha_0'(\mathcal{H})$ bounds $\operatorname{pd} (S/I(\mathcal{H}))$, there exists a non-zero Betti number $b_{i}(S/I(\mathcal{H}))$. It is interesting to know what conditions on $A$ give rise to non-zero Betti numbers of $S/I(\mathcal{H})$ in homological degree $i$. A partial answer to this question was given in \cite[Theorem~3.1]{kimura} where the author developed the concept of strongly disjoint bouquets to give sufficient conditions for non-vanishing Betti numbers of edge ideals of simple graphs.

Although the projective dimension of a sequentially Cohen-Macaulay edge ideal has a closed formula, its other algebraic invariants are not so well understood. It is an open problem for instance to find the regularity of edge ideals of vertex decomposable hypergraphs, which form a subfamily of sequentially Cohen-Macaulay edge ideals. As another example, no explicit combinatorial description is known for the Betti numbers of edge ideals of chordal graphs which are known to be vertex decomposable \cite{Woodroofe}.

Since many tools for studying resolutions of edge ideals involve the edges of the underlying hypergraphs instead of their vertices (Taylor's resolution or simplicial resolutions in general), our characterization of minimal vertex covers may become convenient while studying similar questions mentioned above.


\section{Definitions}
A \emph{simple hypergraph} $\mathcal{H}$ on a finite set $V(\mathcal{H})$ is a family $\mathcal{H}=(\mathcal{E}_1,\dots,\mathcal{E}_d)$ of subsets of $V(\mathcal{H})$ such that
\begin{itemize}
\itemsep-0.2em
\item[(1)] $\mathcal{E}_i\neq \emptyset$ for all $i=1,\dots,d$
\item[(2)] $\bigcup_{i=1}^d\mathcal{E}_i=V(\mathcal{H})$
\item[(3)] $\mathcal{E}_i\subseteq \mathcal{E}_j \Longrightarrow i=j$.
\end{itemize}
The elements of $V(\mathcal{H})$ are called the \emph{vertices} and the sets $\mathcal{E}_1,\dots,\mathcal{E}_d$ are the \emph{edges} of $\mathcal{H}$. We write $E(\mathcal{H})$ for the set of edges of $\mathcal{H}$. If every edge of a simple hypergraph consists of two elements, then it is called a \emph{simple graph}.

A simple hypergraph $\mathcal{K}$ is said to be a \emph{partial hypergraph} of $\mathcal{H}$ if $\mathcal{K}=(\mathcal{E}_j : j\in J)$ for some $J\subseteq \{1,\dots,d\}$. For a set $A\subseteq V(\mathcal{H})$, we define the \emph{partial hypergraph of} $\mathcal{H}$ \emph{on} $A$ as $\mathcal{H}|_A=(\mathcal{E} \in E(\mathcal{H}) : \mathcal{E}\subseteq A )$. The family $\mathcal{H}_A=(\mathcal{E}_j\cap A : 1\leq j \leq d, \mathcal{E}_j\cap A\neq \emptyset)$ is called the \emph{subhypergraph induced by} $A$. Note that $\mathcal{H}|_A$ and $\mathcal{H}_A$ are different hypergraphs. We say that $A$ is \emph{independent} in $\mathcal{H}$ if $\mathcal{E} \nsubseteq A$ for all $\mathcal{E}\in E(\mathcal{H})$.

\begin{definition}[Bouquet]\label{def:bouquet}(Compare to  \cite[Definition~1.7]{zheng})  A \emph{bouquet} is a simple hypergraph ${\mathcal{B}}=(\mathcal{E}_1,\dots,\mathcal{E}_d)$ together with an assigned set of \emph{flowers} $F(\mathcal{B})=\{\ell_1,\dots,\ell_d\}$ such that
\begin{itemize}
\itemsep-0.2em
\item[(1)] $ \bigcap_{i=1}^d\mathcal{E}_i \neq \emptyset $ 
\item[(2)] $\ell_i \in \mathcal{E}_j \Leftrightarrow i=j $, for all $i, j\in\{1,\ldots , d \}$.
\end{itemize}
\end{definition}
The notion of bouquet for simple graphs was introduced by Zheng \cite{zheng} to study resolutions of edge ideals of forests. Our definition generalizes this to arbitrary simple hypergraphs.
Observe that one can assign flowers to a simple hypergraph in different ways to make it a bouquet. However if the bouquet $\mathcal{B}$ is a simple graph with at least two edges, then its flowers are automatically determined.
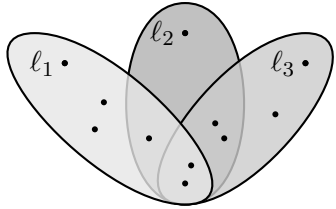
\begin{figure}[ht]
\begin{center}
\begin{tikzpicture}[scale= 0.4]
    \node (v1) at (0,0) {};
    \node (v2) at (-4,4) {};
    \node (v3) at (0,5) {};
    \node (v4) at (4,4) {};
    \node (v5) at (3,2.3) {};
    \node (v6) at (1,2) {};
    \node (v7) at (1.3,1.5) {};
    \node (v8) at (-3,1.8) {};
    \node (v9) at (0.2, 0.6) {};
    \node (v10) at (-1.2, 1.5) {};
    \node (v11) at (-2.7, 2.7) {};

    \begin{scope}[fill opacity=0.8]

\filldraw[fill=black!30][thick] ($(v1)+(0,-0.7)$) 
        to[out=180,in=180] ($(v3) + (0,1)$) 
        to[out=0,in=0] ($(v1) + (0,-0.7)$);
    
\filldraw[fill=black!20][thick] ($(v1)+(0,-0.7)$) 
        to[out=180,in=180] ($(v4) + (0,1)$) 
        to[out=0,in=0] ($(v1) + (0,-0.7)$);
\filldraw[fill=black!10][thick] ($(v1)+(0,-0.7)$) 
        to[out=180,in=180] ($(v2) + (-1,1)$) 
        to[out=0,in=0] ($(v1) + (0,-0.7)$)
         
     ;
   
    \end{scope}

    \foreach \v in {1,2,...,11} {
        \fill (v\v) circle (0.1);
    }

    \fill (v2) circle (0.1) node [left] {$\ell_1$};
    \fill (v3) circle (0.1) node [left] {$\ell_2$};
    \fill (v4) circle (0.1) node [left] {$\ell_3$};

    
\end{tikzpicture}
\end{center}
\caption{A bouquet with flowers $\ell_1, \ell_2, \ell_3$}
\end{figure}

If a partial hypergraph $\mathcal{B}$ of $\mathcal{H}$ is a bouquet, then we simply say that $\mathcal{B}$ is a bouquet of $\mathcal{H}$. Suppose that $\bm{\mathcal{B}}=\{\mathcal{B}_1,\dots,\mathcal{B}_j\}$ is a set of bouquets of $\mathcal{H}$. Then we set
$$F(\bm{\mathcal{B}}):= \bigcup_{i=1}^{j}F(\mathcal{B}_i), \ \ E(\bm{\mathcal{B}}):= \bigcup_{i=1}^{j}E(\mathcal{B}_i)  \text{ and } V(\bm{\mathcal{B}}):= \bigcup_{i=1}^{j}V(\mathcal{B}_i).$$

We call $F(\bm{\mathcal{B}}), E(\bm{\mathcal{B}}) $ and $V(\bm{\mathcal{B}})$ the \emph{flower set}, the \emph{edge set} and the \emph{vertex set} of $\bm{\mathcal{B}}$ respectively. 

\begin{definition}(Compare to \cite[Definition~5.1]{kimura})
\label{semi strongly disjoint hypergraph} A set $\bm{\mathcal{B}}=\{\mathcal{B}_1,\mathcal{B}_2,\dots,\mathcal{B}_j\}$ of bouquets of a simple hypergraph $\mathcal{H}$ is said to be \emph{semi-strongly disjoint} in $\mathcal{H}$ if the following conditions hold.
\begin{itemize}
\itemsep-0.2em
\item[(1)] If $\mathcal{E}\in E(\mathcal{B}_p)$ then $\mathcal{E}\cap F(\mathcal{B}_q)=\emptyset$ for all $q\neq p$.

\item[(2)] $V(\bm{\mathcal{B}})\setminus F(\bm{\mathcal{B}})$ is independent in $\mathcal{H}$.
\end{itemize}
And we set
$$d_{\mathcal{H}}^\prime= \operatorname{max}\{|E(\bm{\mathcal{B}})| : \bm{\mathcal{B}} \text{ is a semi-strongly disjoint set of bouquets of } \mathcal{H}\}$$
or, from Definition~\ref{semi strongly disjoint hypergraph}(1) and Definition~\ref{def:bouquet}(2), equivalently

$$d_{\mathcal{H}}^\prime= \operatorname{max}\{|F(\bm{\mathcal{B}})| : \bm{\mathcal{B}} \text{ is a semi-strongly disjoint set of bouquets of } \mathcal{H}\}.$$
\end{definition}

Note that for a simple graph $G$ the inequality $\alpha_0^\prime (G)\geq d_G^\prime$ was  proved in  \cite[Proposition~2.7]{c-5 free Khosh Moradi}. However $\alpha_0^\prime (G)= d_G^\prime$ was known only for the special case of vertex decomposable graphs (see, \cite[Theorem~3.8]{khosh moradi}). We will see that in fact, the parameters $\alpha_0^\prime (G)$ and $d_G^\prime$ are the same.


\section{Proof of the main result}
\begin{lemma}\label{l: max cardinality min vertex cover number} Let $\mathcal{H}$ be a simple hypergraph. If $\mathcal{K}=\mathcal{H}|_U$ for some $U\subseteq V(\mathcal{H})$, then any minimal vertex cover of $\mathcal{K}$ can be extended to a minimal vertex cover of $\mathcal{H}$. In particular, $\alpha_0^\prime (\mathcal{K})\leq \alpha_0^\prime (\mathcal{H})$.
\end{lemma}
\begin{proof}
Suppose that $V(\mathcal{H})\setminus V(\mathcal{K})=A$ and $C$ is a minimal vertex cover of $\mathcal{K}$. If $A=\emptyset$ then $\mathcal{K}=\mathcal{H}$ and there is nothing to prove. So we assume that $A\neq \emptyset$. Clearly $A\cap C=\emptyset$ and $A\cup C$ covers $\mathcal{H}$. By removing the redundant elements from $A \cup C$, one can get a minimal vertex cover $C' \subseteq A\cup C$ of $\mathcal{H}$. But then $C'\setminus A$ is a minimal vertex cover of $\mathcal{K}$. Since $C'\setminus A\subseteq C$ we get $C'\setminus A=C$ by minimality of $C$. Thus $C\subseteq C'$ is the desired extension.

\end{proof}
\begin{remark} The Lemma above is not necessarily true if $\mathcal{K}$ is an arbitrary partial hypergraph of $\mathcal{H}$. See for example Figures \ref{fig:minipage1} and \ref{fig:minipage2}.
\end{remark}
\begin{figure}[ht]
\begin{minipage}[b]{0.48\linewidth}
\begin{center}
\begin{tikzpicture}
[scale=0.8, vertices/.style={draw, fill=black, circle, inner sep=1.7pt}]
\node[vertices, label=below:{}] (a) at (-2,1) {};
\node[vertices, label=right:{}] (b) at (-1,0) {};
\node[vertices, label=left:{}] (c) at (-2,-1) {};
\node[vertices, label=above:{}] (d) at (2,1) {};
\node[vertices, label=above:{}] (e) at (1,0) {};
\node[vertices, label=above:{}] (f) at (2,-1) {};
\foreach \to/\from in {a/b,b/c,e/d,e/f}
\draw [thick] [-] (\to)--(\from);
\end{tikzpicture}
\end{center}
\caption{A simple graph $K$ with $\alpha_0^\prime (K)=4$}
\label{fig:minipage1}
\end{minipage}
\quad
\quad
\begin{minipage}[b]{0.48\linewidth}
\begin{center}
\begin{tikzpicture}
[scale=0.8, vertices/.style={draw, fill=black, circle, inner sep=1.7pt}]
\node[vertices, label=below:{}] (a) at (-2,1) {};
\node[vertices, label=right:{}] (b) at (-1,0) {};
\node[vertices, label=left:{}] (c) at (-2,-1) {};
\node[vertices, label=above:{}] (d) at (2,1) {};
\node[vertices, label=above:{}] (e) at (1,0) {};
\node[vertices, label=above:{}] (f) at (2,-1) {};
\foreach \to/\from in {a/b,b/c,e/d,e/f, b/e}
\draw [thick] [-] (\to)--(\from);
\end{tikzpicture}
\end{center}
\caption{A simple graph $H$ with $\alpha_0^\prime (H)=3$}
\label{fig:minipage2}
\end{minipage}
\end{figure}

We now prove the main result of this paper.
\begin{theorem}\label{l: equality max min vertex cover versus max semi-strongly disjoint} For any simple hypergraph $\mathcal{H}$, the flower set of a semi-strongly disjoint set of bouquets of $\mathcal{H}$ can be extended to a minimal vertex cover of $\mathcal{H}$. Conversely, for any minimal vertex cover $C$ of $\mathcal{H}$, there exists a semi-strongly disjoint set of bouquets of $\mathcal{H}$ with the flower set $C$. In particular, the equality $\alpha_0^\prime (\mathcal{H})= d_\mathcal{H}^\prime $ holds.
\end{theorem}
\begin{proof} First suppose that $\bm{\mathcal{B}}=\{\mathcal{B}_1,\dots,\mathcal{B}_j\}$ is a semi-strongly disjoint set of bouquets of $\mathcal{H}$. Consider $\mathcal{H}|_{V(\bm{\mathcal{B}})}$, the partial hypergraph of $\mathcal{H}$ on $V(\bm{\mathcal{B}})$. Then $F(\bm{\mathcal{B}})$ is a vertex cover of $\mathcal{H}|_{V(\bm{\mathcal{B}})}$ by the independence of $V(\bm{\mathcal{B}})\setminus F(\bm{\mathcal{B}})$ and it is minimal by condition $(1)$ of Definition \ref{semi strongly disjoint hypergraph}. Thus, by Lemma \ref{l: max cardinality min vertex cover number} the first part of the given statement is verified.

Next, suppose that $C$ is a minimal vertex cover of $\mathcal{H}$. We will construct a set $\bm{\mathcal{B}}$ of semi-strongly disjoint bouquets of $\mathcal{H}$ such that $ F(\bm{\mathcal{B}})= C$. Observe that if $\mathcal{H}$ has any edges of the form $\{u\}$ for some vertex $u$, then $C$ must contain $u$, and every semi-strongly disjoint bouquet can have $\{u\}$ added to it as a bouquet with one edge and $u$ as the flower of that edge. Therefore we may assume that $\mathcal{H}$ has no edges of size one.

 Note that by the minimality of $C$, for every $v\in C$ there exists an edge $\mathcal{E}_v$ of $\mathcal{H}$ such that $\mathcal{E}_v \cap C =\{v\} $. Pick an element $\ell_{1}^1\in C$. Then there exists an edge $\mathcal{E}_{1}^1$ of $\mathcal{H}$ such that $C \cap \mathcal{E}_{1}^1=\{\ell_{1}^1\}$. As $\mathcal{E}_1^1\neq \{\ell_1^1\}$ there exists $r_1\in \mathcal{E}_{1}^1\setminus\{\ell_{1}^1\}$. Suppose that $\ell_{1}^1,\ell_{2}^1,\dots,\ell_{d_1}^1$ are the elements of $C$ that satisfy the property
 \begin{equation}\label{property for basis case}\text{there exists } \mathcal{E}_{i}^1\in E(\mathcal{H}) \text{ such that } \mathcal{E}_{i}^1\cap C =\{\ell_{i}^1\} \text{ and } r_1\in \mathcal{E}_{i}^1 
 \end{equation}
for every $1\leq i\leq d_1$. Let $\mathcal{E}_1^1,\dots,\mathcal{E}_{d_1}^1$ be chosen fixed edges that satisfy the property above. Consider the partial hypergraph $\mathcal{B}_1=(\mathcal{E}_{1}^1,\dots,\mathcal{E}_{d_1}^1)$ of $\mathcal{H}$ with the assigned flowers $\ell_{1}^1,\dots,\ell_{d_1}^1$. 

Now if $F(\mathcal{B}_1)=C$, then $\bm{\mathcal{B}}=\{\mathcal{B}_1\}$ and we are done. Otherwise we keep constructing new bouquets inductively as follows. Suppose that we have semi-strongly disjoint bouquets \{$\mathcal{B}_1,\dots,\mathcal{B}_t$\} such that $\cup_{i=1}^tF(\mathcal{B}_i)$ is a proper subset of $C$. Pick an element $\ell_{1}^{t+1}\in C\setminus\cup_{i=1}^tF(\mathcal{B}_i)$ and an edge $\mathcal{E}_{1}^{t+1}$ of $\mathcal{H}$ such that $\mathcal{E}_1^{t+1}\cap C=\{\ell_{1}^{t+1}\}$. Fix $r_{t+1}\in\mathcal{E}_1^{t+1}\setminus \{\ell_1^{t+1}\}$ and let $\ell_{1}^{t+1},\ell_{2}^{t+1},\dots,\ell_{d_{t+1}}^{t+1}$ be the elements of $C \setminus \cup_{i=1}^tF(\mathcal{B}_i) $ that satisfy the property
\begin{equation}\label{property for inductive case}\text{there exists } \mathcal{E}_{i}^{t+1}\in E(\mathcal{H}) \text{ such that } \mathcal{E}_{i}^{t+1}\cap C =\{\ell_{i}^{t+1}\} \text{ and } r_{t+1}\in \mathcal{E}_{i}^{t+1} 
\end{equation}
for every $1\leq i\leq d_{t+1}$. Let $\mathcal{E}_{1}^{t+1},\dots,\mathcal{E}_{d_{t+1}}^{t+1}$ be chosen fixed edges that satisfy the property above. Consider the partial hypergraph $\mathcal{B}_{t+1}=(\mathcal{E}_{1}^{t+1},\dots,\mathcal{E}_{d_{t+1}}^{t+1})$ of $\mathcal{H}$ as a bouquet with flowers $\ell_{1}^{t+1},\ell_{2}^{t+1},\dots,\ell_{d_{t+1}}^{t+1}$. We shall show that $\{\mathcal{B}_1,\dots,\mathcal{B}_{t+1}\}$ is semi-strongly disjoint. Condition $(1)$ of Definition \ref{semi strongly disjoint hypergraph} clearly holds by construction. To see that the second condition holds, observe that 
$$ V(\{\mathcal{B}_1,\dots,\mathcal{B}_{t+1}\})\setminus F(\{\mathcal{B}_1,\dots,\mathcal{B}_{t+1}\}) = V(\{\mathcal{B}_1,\dots,\mathcal{B}_{t+1}\})\setminus C $$
is independent in $\mathcal{H}$ since $C$ is a vertex cover, so every edge intersects $C$.

Having verified that this construction yields semi-strongly disjoint bouquets at every step, we know that it will terminate as $\mathcal{H}$ has finitely many vertices. In that case, $C=\cup_{i=1}^pF(\mathcal{B}_i)$ for some $p\geq 1$ and $\bm{\mathcal{B}}=\{\mathcal{B}_1,\dots,\mathcal{B}_p\}$ is as desired.

\end{proof}

The following is an immediate consequence of Theorem \ref{l: equality max min vertex cover versus max semi-strongly disjoint}.
\begin{corollary}
Given a simple hypergraph $\mathcal{H}$, we have the following statements. 
\begin{itemize}
\itemsep-0.2em
\item[(1)] If $\bm{\mathcal{B}}$ is a semi-strongly disjoint set of bouquets of $\mathcal{H}$ such that $\left \vert F(\bm{\mathcal{B}})\right \vert =d_\mathcal{H}^\prime$, then $F(\bm{\mathcal{B}})$ is a minimal vertex cover of $\mathcal{H}$ of maximum cardinality.
\item[(2)] If $C$ is a minimal vertex cover of $\mathcal{H}$ of maximum cardinality, then there exists a semi-strongly disjoint set $\bm{\mathcal{B}}$ of bouquets of $\mathcal{H}$ such that $F(\bm{\mathcal{B}})=C$ and $\left\vert F(\bm{\mathcal{B}})\right \vert = d_\mathcal{H}^\prime . $
\end{itemize}
\end{corollary}

Lastly, we state the following corollary regarding projective dimension of edge ideals.
\begin{corollary} Let $I$ be the edge ideal of a simple hypergraph $\mathcal{H}$. Then $\operatorname{pd} (S/I) \geq d_\mathcal{H}^\prime$.
\end{corollary}
\begin{proof}
By \cite[Corollary~3.33]{morey villarreal}, we have $\operatorname{pd} (S/I) \geq \alpha_0^\prime (\mathcal{H}). $ Hence the proof is immediate from Theorem \ref{l: equality max min vertex cover versus max semi-strongly disjoint}.
\end{proof}

\section{Computation of $d_{\mathcal{H}}'$ and further examples of bouquets}
If $G$ is a graph on $n$ vertices, then $d_G'\leq n-1$ since for every bouquet $\mathcal{B}=(\mathcal{E}_1,\ldots ,\mathcal{E}_d)$, the intersection $\cap_{i=1}^d\mathcal{E}_i$ is non-empty. In the next example, we will use this observation to compute $d_G'$.

\begin{example}
	Let $G$ be the graph whose edge ideal is $I(G)=(ad, ae, af, bd, be, bf, ce, cf)$. Notice that $G$ is a bipartite graph on $6$ vertices with bipartition $(X, Y)$ where $X=\{a,b,c\}$ and $Y=\{d,e,f\}$. From the discussion above we know that $d_G'\leq 5$. Observe that $G$ has no vertex of degree $5$. So, $G$ cannot have semi-strongly disjoint bouquets with $5$ flowers and, $d_G'\leq 4$.
	
	Now let $\mathcal{B}_1=(\{a,d\}, \{b,d\})$ and $\mathcal{B}_2=(\{c,e\}, \{c,f\} )$. Then $\bm{\mathcal{B}}=\{B_1, B_2\}$ is semi-strongly disjoint in $G$ as $\{c,d\}$ is independent. Therefore $d_G'=4$.
\end{example}

\begin{example}
	In \cite{hoefel} Hoefel and Mermin defined \textit{supernovas} to characterize Gotzmann squarefree monomial ideals. A $d$-dimensional simplicial complex $\Delta$ is called a \textit{supernova} if there exists a chain of faces $\emptyset \subset F_0 \subset F_1 \subset \cdots \subset F_{d-1}$ such that every $i$-dimensional facet of $\Delta$ contains the $(i-1)$-dimensional face $F_{i-1}$. Observe that the facets of the supernova $\Delta$ have a nonempty intersection as $F_0$ is contained in every facet of $\Delta$. Moreover, if $G$ is an $i$-dimensional facet, then the vertex in $G\setminus F_{i-1}$ is a free vertex of $\Delta$. Therefore every supernova, when considered as a hypergraph,  is a bouquet. 
	
	On the other hand, not every bouquet is a supernova. For example, let $\mathcal{B}$ be a bouquet with facets (or edges) $ \{a,b,d\},\{a,b,c,f\},\{b,c,e\}$ and flowers $d,e,f$. Then $\mathcal{B}$ is not a supernova. To see this, notice that if $G$ and $H$ are distinct facets of a supernova with the same dimension, then $G\setminus F$ is a $0$-dimensional face of the supernova. But we see that $\{a,b,d\}\setminus \{b,c,e\}=\{a,d\}$ is a $1$-dimensional face of $\mathcal{B}$. 
\end{example}

\begin{example}
	Let $\mathcal{H}$ be the hypergraph whose edge ideal is $I(\mathcal{H})=(abc, bde, bce, cef)$. Suppose that a semi-strongly disjoint bouquets $\bm{\mathcal{B}}$ of $\mathcal{H}$ contains the edge $\{b,c,e\}$ where $b$ is the assigned flower of that edge. Then the edges $\{a,b,c\}, \{b,d,e\}\notin E(\bm{\mathcal{B}})$ by condition (1) of Definition~\ref{semi strongly disjoint hypergraph}. By symmetry, we can conclude that if $\{b,c,e\}$ is an edge of semi-strongly disjoint bouquets of $\mathcal{H}$, then at least $2$ of the other edges of $\mathcal{H}$ are not edges of the semi-strongly disjoint bouquets. Therefore $d_{\mathcal{H}}'\leq 3$.
	
	From the argument above, the only case a semi-strongly disjoint bouquets $\bm{\mathcal{B}}$ can have $3$ edges is when $E(\bm{\mathcal{B}})=\{\{a,b,c\}, \{b,d,e\}, \{c,e,f\} \}$. This requires $F(\bm{\mathcal{B}})=\{a,d,f\}$. But then $V(\bm{\mathcal{B}})\setminus \{a,d,f\}$ is not independent in $\mathcal{H}$. So this case is not possible and, it follows that $d_{\mathcal{H}}'\leq 2$.
	
	Now observe that $\mathcal{H}$ has a semi-strongly disjoint bouquets with $2$ edges. Indeed, if $\mathcal{B}_1=(\{a,b,c\})$ is a bouquet with flower $a$ and, $\mathcal{B}_2=(\{c,e,f\})$ is a bouquet with flower $e$, then $\{\mathcal{B}_1, \mathcal{B}_2\}$ is semi-strongly disjoint. Thus $d_{\mathcal{H}}'= 2$.
\end{example}

\vskip0.3in
\noindent {\bf \large Acknowledgments:} I would like to thank my supervisor Sara Faridi for very useful discussions which improved the presentation of this paper. I also thank Fahimeh Khosh-Ahang Ghasr for pointing out some flaws in the early version of this work.

\end{document}